%% file: 0Template.tex
\documentclass[a4paper,twoside,11pt]{article}
\usepackage{a4wide, fancyhdr, amsmath, amssymb, mathtools, yfonts}
\usepackage{mathrsfs}
\usepackage{graphicx}
\usepackage{tikz}
\usepackage[all]{xy}
\usepackage[utf8]{inputenc}
\usepackage{amsthm}
\usepackage[english]{babel}
\usepackage{chngcntr}
\usepackage{ifthen}
\usepackage{calc}
\usepackage{hyperref}
\usepackage{authblk}
\usepackage{tikz-cd}
\usepackage{xfrac}
\usepackage{multicol}
\usepackage[bb=dsserif]{mathalpha}
\numberwithin{equation}{section}
\usepackage{enumerate}
\usepackage{listings}
\usepackage{pdfpages}

\newcommand{\Z}{\mathbb{Z}}
\newcommand{\Q}{\mathbb{Q}}

\newcommand{\QP}{\mathbb{Q}_{p}}
\newcommand{\ZP}{\mathbb{Z}_{p}}
\newcommand{\FP}{\mathbb{F}_{p}}

\newcommand{\OO}{\mathcal{O}}

\newcommand{\sums}[1]{\sum_{\substack{#1}}}

\newtheorem{lemma}{Lemma}[section]
\newtheorem{theorem}[lemma]{Theorem}
\newtheorem{proposition}[lemma]{Proposition}
\newtheorem{corollary}[lemma]{Corollary}
\newtheorem{mydef}[lemma]{Definition}

\newtheorem{remark}[lemma]{Remark}

\def\mylist#1 {\ifx!#1\else\makebox[4em][r]{#1} \expandafter\mylist\fi}
\title{\vspace{-\baselineskip}\sffamily\bfseries Local solubility of ternary cubic forms}
\author[1]{Golo Wolff\thanks{ETH Zürich, Switzerland, gwolff@student.ethz.ch}}
\date{October 20, 2024}

\begin{document}
\maketitle

\begin{abstract}
We consider cubic forms
\[
\phi_{a,b}(x,y,z) = ax^3 + by^3 - z^3
\]
with coefficients $a,b \in \Z$. We give an asymptotic formula for how many of these forms are locally soluble everywhere, i.e. we give an asymptotic formula for the number of pairs of integers $(a,b)$ that satisfy $1 \leq a \leq A$, $1 \leq b \leq B$ and some mild conditions, such that $\phi_{a,b}$ has a non-zero solution in $\QP$ for all primes $p$.
\end{abstract}
\input{1Introduction}
\input{2Tools}
\input{3Main}

\appendix
\input{5appendix}
\end{document}

%% file: 1Introduction.tex
\section{Introduction}
\label{introduction}
The idea for this paper is based on a paper \cite{FRIW} by John Friedlander and Henryk Iwaniec. In their paper the authors give an asymptotic formula for the number of ternary quadratic forms $ax^2 + by^2 - z^2$ with integer coefficients $1\leq a \leq A$, $1\leq b \leq B$ that are locally soluble everywhere and thus have a non-zero rational solution due to the Minkowski local-global principle for indefinite quadratic forms. They restrict to counting pairs $(a,b)$ which are square-free, coprime and odd. We intend to modify their argument and prove a similar result in the cubic case.

Our setup is the following. Let $\phi_{a,b}(x,y,z) = ax^3 + by^3 - z^3$ have coefficients $a,b \in \Z$ and let $N(A,B)$ denote the number of integer pairs $(a,b)$ that satisfy the following conditions:
\begin{itemize}
    \item $1 \leq a \leq A$, $1 \leq b \leq B$,
    \item $\phi_{a,b}$ has a non-zero solution locally everywhere,
    \item $a,b$ are square-free, coprime and not divisible by three.
\end{itemize}
Then we get the following asymptotic. 
\begin{theorem}
\label{AsymptoticFormula}
    Let $0< \delta < 1$. For $A,B \geq \exp ((\log AB)^{\delta})$ we have
    \[
    N(A,B) = c_2 \cdot \frac{AB}{(\log A)^{\frac{1}{3}}(\log B)^{\frac{1}{3}}} \left(1 + \OO \bigg(\frac{1}{\log A} +\frac{1}{\log B}\bigg)\right),
    \]
    where $c_2$ is the constant
    $$
    c_2
    =
    \Big(\Gamma\Big(\frac{2}{3}\Big)\Big)^{-2}
    \prod_{p}
    \left(1+\frac{2}{3^{\omega_1(p)}p}\right)
    \left(1- \frac{1}{p}\right)^{\frac{4}{3}}.
    $$
\end{theorem}
For the cubic case of this paper we have no analogue to the local-global principle used by Friedlander and Iwaniec in the quadratic case. With the tools of our paper it is not possible to infer that $\phi_{a,b}$ has a rational solution, even if it is locally soluble everywhere. As an immediate consequence of Theorem \ref{AsymptoticFormula}, we only achieve an upper bound on the number of forms that have a global solution.
\begin{corollary}
    Let $0<\delta<1$ and $A,B \geq \exp ((\log AB)^{\delta})$. Then the number of cubic forms $\phi_{a,b}$ that have a global solution is bounded from above by 
    $$
    c_3 \cdot \frac{AB}{(\log A)^{\frac{1}{3}}(\log B)^{\frac{1}{3}}} \left(1 + \frac{1}{\log A} +\frac{1}{\log B}\right),
    $$ 
    if we assume the same restrictions on the coefficients $a,b$ as in the definition of $N(A,B)$. The leading constant $c_3$ may depend on $\delta$.
    
    Furthermore
    $$
    \limsup_{A,B \rightarrow \infty}\; N(A,B)\bigg(\frac{AB}{(\log A)^{\frac{1}{3}}(\log B)^{\frac{1}{3}}}\bigg)^{-1} \leq c_2.
    $$
\end{corollary}
Note that the condition $A,B \geq \exp ((\log AB)^{\delta})$ intuitively means that $A,B$ are not extremely different in size. It implies, for example, that $(\log B)^{\sfrac{1}{\delta}} \geq \log A \geq (\log B)^{\delta}$ and it is mainly required for the way that we have bounded the error terms in the proof of Theorem \ref{AsymptoticFormula}.

\section*{Acknowledgements}
I thank my supervisor Peter Koymans, not only for spending many hours on helping me with this paper, but also for teaching me analytic number theory and mathematics in general.

I wish to acknowledge everybody. It is society as a whole which provides us with the possibility for abstract mathematical research. 

%% file: 2Tools.tex
\section{Tools}
This section contains two tools used in the proof of Theorem \ref{AsymptoticFormula} and its preliminary lemmas.

The first tool is a large sieve result taken from \cite[Prop. 4.3]{KoymansRome} and it requires a setup that will be repeated here, because we have to verify this setup in the proof of Lemma \ref{DoubleOscillation}.

Let $K$ be a number field and let $l$ be a prime number. If $\mathfrak{f}$ is an ideal, we write $S_{\mathfrak{f}}$ for the subset of $\alpha \in \mathcal{O}_K$ coprime with $\mathfrak{f}$. We also write $N(w)$ for the absolute norm of an element. Let $M \geq 1$ be an integer. Suppose that we are given a map $\gamma : S_{M \mathcal{O}_K} \times S_{M \mathcal{O}_K} \longrightarrow \{0\} \cup \{\zeta_l^i : i = 0,...,l-1\}$ and a subset $A_{\text{bad}}$ of $\Z_{\geq 0}$ satisfying the following properties.

(P1) Multiplicativity: we have
\begin{gather*} \allowdisplaybreaks
\gamma(w, z_1z_2) = \gamma(w, z_1)\gamma(w, z_2) \text{ for all } w, z_1 \text{ and } z_2,
\\
\gamma(w_1w_2, z) = \gamma(w_1,z)\gamma(w_2,z)
\text{ for all } w_1, w_2 \text{ and } z.
\end{gather*}

(P2) Periodicity: if $z_1,z_2,w \in S_{M\mathcal{O}_K}$ satisfy $z_1 \equiv z_2$ mod $N(w)$ and $z_1 \equiv z_2$ mod $M$, then we have $\gamma(w,z_1) = \gamma(w,z_2)$. Furthermore, if $N(w) \notin A_{\text{bad}}$, then we have 
\[
\sum_{\substack{
\xi \bmod M N(w)\\
\gcd(\xi,M) = (1)
}}
\gamma(w,\xi)
=
0
\]

(P3) Bad count: we have 
\[
\sum_{\substack{
n \in A_{\text{bad}}\\
n \leq X
}}
1
\leq
C_1X^{1-C_2}
\]
for some absolute constants $C_1 > 0$ and $0 < C_2 < 1$.

Decompose $\mathcal{O}_K^{\ast} = T \oplus V$, where $T$ is torsion and $V$ is free. Such a decomposition is not unique, but we will fix one such decomposition. Fix a fundamental domain $\mathcal{D} \subseteq \mathcal{O}_K$ as in \cite[Section 3.3]{KoymansMilovic} for the action\footnote{For the number field $\Q[\zeta_3]$ required in Lemma \ref{DoubleOscillation} $V$ is trivial, hence there is only one choice for $\mathcal{D}$. The reader may skip looking up \cite[Section 3.3]{KoymansMilovic}.} of $V$ on $\mathcal{O}_K$. 

We will consider bilinear sums of the type
\[
B(X,Y, \delta, \epsilon, t_1, t_2) 
= \sum_{\substack{
w \in t_1\mathcal{D}(X) \\
w \equiv \delta \bmod M
}}
\sum_{\substack{
z \in t_2\mathcal{D}(Y) \\
z \equiv \epsilon \bmod M
}}
\alpha_w\beta_z\gamma(w,z)
\]
where $(\alpha_w)_w$ and $(\beta_z)_z$ are sequences of complex numbers bounded in absolute value by 1, $\delta$ and $\epsilon$ are invertible congruence classes modulo $M$, $t_1$ and $t_2$ are fixed elements of $T$ (so $t_i\mathcal{D}$ is a translate of the fundamental domain) and $X, Y \geq 2$ are real numbers. Here we use the notation $(t_i\mathcal{D})(X)$ for the subset of $\alpha \in t_i\mathcal{D}$ with $N(\alpha)\leq X$. Then we have the following proposition, taken from \cite[Prop 4.3]{KoymansRome}.
\begin{proposition}\label{AbstractLargeSieve}
    Assume $X\leq Y$. Then we have
    \[
    |B(X,Y, \delta, \epsilon, t_1, t_2)|
    \ll
    (X^{-\frac{C_2}{3n}}+Y^{-\frac{1}{6n}})XY(\log XY)^{C_K},
    \]
    where $n = [K : \Q]$ and $C_K$ is a constant depending only on $K$. The implied constant depends only on $K$, $M$ and the constants $C_1$, $C_2$.
\end{proposition}
The second tool is a theorem on sums of multiplicative functions taken from \cite[Theorem 13.2, p. 134]{KOUKOU}.
\begin{theorem}\label{Theorem 13.2 KOUKOU}
    Let $Q > 2$ be a parameter and let $f$ be a multiplicative function such that there exists $\alpha \in \mathbb{C}$ with
    \begin{gather}
    \label{KOKOU condition}
    \sum_{p \leq x} f(p)\log(p) 
    =
    \alpha x + \OO_A\left(\frac{x}{\log(x)^A}\right) \quad \quad (x \geq Q)
    \end{gather}
    for all $A > 0$. Moreover, suppose for all $n$ that $|f(n)| \leq \tau_k(n)$ for some positive real number $k$. Fix $\epsilon > 0$ and $J \in \Z_{\geq 1}$. Then we have
    \[
    \sum_{n \leq x} f(n) 
    = 
    x \sum_{j = 0}^{J-1}
    c_j
    \frac{\log(x)^{\alpha - j - 1}}{\Gamma(\alpha-j)}
    +
    \OO\left(\frac{x(\log Q)^{2k+J-1}}{(\log x)^{J+1-Re(\alpha)}}\right)
    \]
    for $x \geq \exp{((\log Q)^{1+\epsilon})}$ and some explicit constants $c_j$. The implied constant depends at most on $k,J, \epsilon$ and the implied constant in equation \eqref{KOKOU condition} for $A$ large enough in terms of $k, J $ and $\epsilon$ only. Furthermore, we have
    \[
    c_0 = \prod_{p}\left(1 + \frac{f(p)}{p} + \frac{f(p^2)}{p^2} + ...\right)\left(1- \frac{1}{p}\right)^\alpha
    \]
    and $c_j \ll_{j,k} (\log Q)^{j+2k}$.
\end{theorem}
\begin{remark}\label{convention}
    Note that the above theorem is typeset under the convention that ``$\frac{1}{\infty} = 0$". If for example $\alpha \in \Z_{\leq 0}$, then poles of the gamma function appear in the denominator of every summand of the main term. In this case we consider the sum to be empty and the main term to vanish.
\end{remark}

%% file: 3Main.tex
\section{Proof of the theorem}
The proof of the theorem will need four lemmas. Two of the lemmas detect solutions of $\phi_{a,b}$ in $\QP$, the other two lemmas are character oscillation results.
\subsection{Lemmas to detect solutions}
The following two lemmas answer the question how to detect a solution of $\phi_{a,b}$ in $\QP$. The first lemma implies that for all primes $p \nmid ab$, $\phi_{a,b}$ has a solution in $\QP$, except for the special case $p=3$. The second lemma states that for primes $p \mid ab$ we can express the solubility of $\phi_{a,b}$ in $\QP$ through cubic characters.
\begin{lemma}
\label{p nmid ab}
    Let $(a,b)$ be any pair of integers with $3 \nmid ab$. Then $\phi_{a,b}$ has a non-zero solution
    \begin{enumerate}[i)]
        \item in $\QP$ if $p \nmid ab$ and $p \neq 3$,
        \item in $\Q_3$ if and only if $(a,b)$ is congruent modulo 27 to an ``admissible pair" from the ``list of admissible pairs" compiled in the appendix \ref{List of admissable pairs}.
    \end{enumerate}
\end{lemma}
\begin{proof}[Proof of Lemma \ref{p nmid ab}] We first prove statement $i)$ for primes $p \geq 11$ and afterwards for the smaller primes. The strategy is to show that $\phi_{a,b}$ has a solution in $\FP$ that can be lifted to a solution in $\ZP$ with Hensel's lemma \cite[Theorem 2.1]{conrad}.
    
    For $p \nmid ab$ and $p \geq 11$, $\phi_{a,b}$ is an non-singular cubic curve of genus 1, so we may use the Hasse-Weil bound \cite{WeilBound} to infer that there are at least 4 projective solutions over $\FP$. For the cubic $\phi_{a,b}$ this implies that there is at least one solution $(x_0, y_0, z_0) \in \FP^3$ with $z_0 \neq 0$. Pick a lift of this triple $(\Tilde{x}, \Tilde{y}, \Tilde{z}) \in \ZP$ and define $f(z) \coloneqq \phi_{a,b}(\Tilde{x}, \Tilde{y}, z) = a\Tilde{x}^3 + b\Tilde{y}^3 - z^3 \in \ZP[z]$. Note that by definition $f(z)$ reduces to $\phi_{a,b}(x_0,y_0,z)$ in $\FP$, so we have $z_0 \not \equiv 0$ mod $p$, $f(z_0) \equiv 0$ mod $p$ and $f'(z_0) = -3 z_0^2 \not \equiv 0$ mod $p$. Thus we may apply Hensel's lemma \cite[Thm 2.1]{conrad} to lift $z_0$ to a non-zero solution of $f(z) = \phi_{a,b}(\Tilde{x}, \Tilde{y}, z)$ in $\ZP$.
    
    If $p \nmid ab$ and $p = 2,5,7$, proving that $\phi_{a,b}$ has a solution in $\QP$ is similar. Instead of the Hasse-Weil bound, we use a small computer program (code in Appendix \ref{List of admissable pairs}) to check all possible values of $x,y,z \in \FP$ until we find a solution of $\phi_{a,b}$ in $\FP$. Indeed for all pairs $(a,b) \in \FP^2$ we find a solution of $\phi_{a,b}$ with at least one non-zero coordinate and it can then be lifted to a solution in $\ZP$ with Hensel's lemma similarly to the previous case.
    
    We now turn to statement $ii)$. For a solution in $\mathbb{Q}_3$ we cannot apply the same version of Hensel's lemma as in the previous cases, as the derivative of $\phi_{a,b}$ is zero modulo 3. We use a stronger version of Hensel's lemma for this case and will prove the following: $\phi_{a,b}$ has a solution in $\Z_3$ if and only if $\phi_{a,b}$ has a solution mod 27 with at least one coordinate not congruent to 0 modulo 3. We will call the pairs $(a,b)$ such that $\phi_{a,b}$ has a solution mod 27 ``admissible pairs".
    
    Suppose $\phi_{a,b}$ has a solution $(x_0,y_0,z_0)$ mod 27 with at least one coordinate that is not congruent to 0 mod 3. Assume that $z_0 \not \equiv 0$ mod 3, in the other cases the argument works analogously. Lift the triple $(x_0,y_0,z_0)$ to $(\Tilde{x}, \Tilde{y}, \Tilde{z}) \in \Z_3$ and define $f(z) \coloneqq \phi_{a,b}(\Tilde{x}, \Tilde{y}, z) = a\Tilde{x}^3 + b\Tilde{y}^3 - z^3 \in \Z_3[z]$. Then $z_0$ is a solution mod 27 of $f(z) \equiv ax_0^3 + by_0^3 - z^3 \bmod 27$ (and also mod 3). Using these properties of $z_0$, we obtain $|f(z_0)|_3 \leq \frac{1}{27} < \frac{1}{9} = |3z_0|_3^2 = |f'(z_0)|_3^2$. Thus we may apply a stronger version of Hensel's lemma \cite[Thm 4.1]{conrad} to lift the solution $z_0$ from $\mathbb{F}_3$ to a solution of $f(z)$ in $\Z_3$.
    
    Conversely, if $\phi_{a,b}$ has a solution in $\Z_3$ we may scale it by a power of 3 such that all coordinates have non-negative valuation but at least one them has valuation 0. The scaled triple remains a solution of $\phi_{a,b}$ because $\phi_{a,b}$ is homogeneous. The triple then reduces to a solution of $\phi_{a,b}$ mod 27 and the coordinate with valuation 0 is not congruent to 0 mod 3. 
\end{proof}
    In the appendix \ref{List of admissable pairs} we have compiled a ``list of admissible pairs". $\phi_{a,b}$ has a solution mod 27 with one coordinate not congruent to 0 mod 3 if and only if $(a,b)$ is congruent modulo 27 to one of the pairs given in that list. The list was compiled with a computer program that simply iterates through all possible values of $x,y,z$ mod 27 until it finds a solution to $\phi_{a,b}(x,y,z) = ax^3+by^3-z^3$ that meets the requirement.
\begin{lemma}
\label{p mid ab}
\begin{enumerate}[i)]
    \item Let $(a,b)$ be a pair of coprime, square-free integers with $3 \nmid ab$. 
    
    For all primes $p \equiv 2$ mod 3 with $p \mid ab$, $\phi_{a,b}$ has a solution in $\QP$.
    \item For all primes $p \equiv 1$ mod 3 we can construct a cubic character $\chi_p$ of modulus $p$ that satisfies the following:

    Let $(a,b)$ be a pair of coprime, square-free integers with $3 \nmid ab$. If $p \mid a$, we have
    \[
    \frac{1}{3}(1+\chi_p(b)+\chi_p^2(b)) = 
    \begin{cases}
        1 & \text{if } \phi_{a,b} \text{ has a non-zero solution in } \QP \\
        0 & \text{otherwise.}
    \end{cases}
    \]
    Analogously if $p \mid b$,
    \[
    \frac{1}{3}(1+\chi_p(a)+\chi_p^2(a)) = 
    \begin{cases}
        1 & \text{if } \phi_{a,b} \text{ has a non-zero solution in } \QP \\
        0 & \text{otherwise.}
    \end{cases}
    \]
\end{enumerate}
\end{lemma}
\begin{proof}[Proof of Lemma \ref{p mid ab}] Let $(a,b)$ be a pair of coprime, square-free integers with $3 \nmid ab$. Assume for this proof that $p \mid a$ and thus by co-primality that $p \nmid b$. The analogue argument works for $p \mid b$, $p \nmid a$. We will first show that $\phi_{a,b}$ has a solution $(x_1,y_1,z_1)$ in $\QP$ if and only if $\phi_{a,b} \equiv by^3-z^3$ mod $p$ has a solution $(x_0,y_0,z_0)$ in $\FP$ that is not one of the trivial solutions $(x_0, 0, 0)$. Afterwards we will give criteria for the existence of a non-trivial solution mod $p$ and construct the cubic character $\chi_p$.

    If $\phi_{a,b}$ has a solution $(x_1,y_1,z_1)$ in $\QP$, we can assume that - after scaling by a power of $p$ - the valuations of $x_1,y_1,z_1$ are non-negative but at least one valuation is zero. Furthermore comparing the valuations of the terms on both sides of the equation $ax_1^3+by_1^3-z_1^3 = 0$ shows that at least one of $y_1$ and $z_1$ need to have valuation zero. Thus $(x_1,y_1,z_1)$ reduces to a non-trivial solution of $\phi_{a,b} \equiv by^3-z^3$ mod $p$.

    Conversely if $\phi_{a,b} \equiv by^3 - z^3$ mod $p$ has a solution $(x_0,y_0,z_0)$ in $\FP$ with $y_0 \neq 0$ or $z_0 \neq 0$, it can be lifted to $\ZP$ with Hensel's Lemma \cite[Thm 2.1]{conrad}, as in the proof of $i)$ in the previous lemma.
    
    In $\FP$, $\phi_{a,b}(x,y,z) \equiv by^3 - z^3$ has a non-trivial solution (a solution that is not of the form $(x_0,0,0)$) $\iff$ $z^3 = by^3$ has a solution in $\FP$ with $z\neq0$ $\iff$ $b$ is a cube in $\FP$. For primes $p \equiv 2$ mod 3 this is always the case because for these primes cubing is an isomorphism of $\FP$. This proves $i)$. For primes $p \equiv 1$ mod 3 we will construct a cubic character $\chi_p$ which satisfies 
    \begin{equation}
    \begin{cases}
    \label{haha}
        \chi_p(b) = 1, & \text{if $b$ is a cube mod $p$} \\
        \chi_p(b) \in \{ \zeta_3, \zeta_3^2\}, & \text{otherwise}.
    \end{cases}
    \end{equation}
    Then $\frac{1}{3}(1+\chi_p(b)+\chi_p^2(b))$ is the indicator function of ``$b$ is a cube modulo $p$", i.e. of ``$\phi_{a,b}$ has a solution in $\QP$".
    
    For primes $p \equiv 1$ mod 3 there are exactly two characters that satisfy the conditions in \eqref{haha} and they can be constructed as follows. The multiplicative group $G \coloneqq \FP^\ast$ is a cyclic group of order $3k$ and thus has a subgroup $H$ of index 3 (if $g$ generates $G$, then $g^3$ is a generator of $H$). The coset $1 \cdot H$ in $\sfrac{G}{H}$ contains all cubes modulo $p$ and thus we may define a character that maps elements of $H$ to $1 \in \mathbb{C}$ and the elements of the other two cosets to $\zeta_3$ and $\zeta_3^2$ respectively. The two characters obtained this way are of order 3, conductor $p$ and complex conjugates of each other.
\end{proof}
\subsection{Lemmas on sums of characters}
As just shown in the proof above, there are precisely two characters for each prime $p \equiv 1$ mod 3 that satisfy property $ii)$ of Lemma \ref{p mid ab}. We now make a choice among these two for the rest of this paper.

Any prime $p \equiv 1 \bmod 3$ splits completely in $\Q[\zeta_3] = \Q[\frac{-1+\sqrt{-3}}{2}] = \Q[\sqrt{-3}]$, i.e. $p = \pi_p \overline{\pi_p}$ with $\pi_p \in \Z[\zeta_3]$ prime and $\text{im}(\pi_p) > 0$. Note that the 3rd power residue symbols 
\[\left(\frac{\cdot}{\pi_p}\right)_{\Z[\zeta_3],3} \text{ and } \left(\frac{\cdot}{\overline{\pi_p}}\right)_{\Z[\zeta_3],3}\]
are cubic characters when interpreted as a function in the top argument and restricted to the rational integers. Furthermore these cubic characters have modulus $p$, map integers that are cubes modulo $p$ to 1 and are distinct. We conclude that on the rational integers, these two power residue symbols correspond to the two possible choices for $\chi_p$ above.
\begin{mydef}
    For rational primes $p \equiv 1 \bmod 3$ define $\chi_p$ to be the cubic character 
    $$\chi_p(\cdot) \coloneqq \left(\frac{\cdot}{\pi_p}\right)_{\Z[\zeta_3],3}: \Z \mapsto \mathbb{C} \; .$$
    Here $\left(\frac{\,\cdot\,}{\,\cdot\,}\right)_{\Z[\zeta_3],3}$ is the 3rd power residue symbol and $\pi_p \in \Z[\zeta_3]$ is the prime over $p = \pi_p \overline{\pi_p}$ with $\text{im}(\pi_p) > 0$. Note that $\chi_p$ satisfies $ii)$ of Lemma \ref{p mid ab}.

    Multiplicatively extend the above definition to square-free integers whose prime factors are all equivalent to 1 mod 3, i.e. for square-free $d \in \Z$ with $p \,|\, d \implies p \equiv 1$ mod 3, define
    $$\pi_d \coloneqq \prod_{p|d} \pi_p \quad \text{ and } \quad \chi_d \coloneqq \prod_{p|d} \chi_p$$
    in the same way that the power residue symbol can be extended
    $$\chi_d(\cdot) = \left(\frac{\cdot}{\pi_d}\right)_{\Z[\zeta_3],3} \coloneqq \prod_{p|d} \left(\frac{\cdot}{\pi_p}\right)_{\Z[\zeta_3],3}. $$
    Lastly we ease notation a bit and simply write $\left(\frac{\,\cdot\,}{\,\cdot\,}\right)_{3}$ for the cubic power residue symbol.
\end{mydef}
For the proof of the main theorem we need character oscillation results. The first such result is a ``double oscillation result", derived from Proposition \ref{AbstractLargeSieve}.
\begin{lemma}
\label{DoubleOscillation}
    Let $s,t \in \{0,1,2\}$ but $(s,t) \neq (0,0)$. Then there exist constants $0 < \kappa_1 < 1$ and $\kappa_2 > 0$ such that
    \[
    \bigg | \sum_{
    \substack{
    1 < d \leq A \\
    p|d \implies \\
    p \equiv 1 \bmod 3}}
    \sum_{
    \substack{
    1 < e \leq B \\
    p|e \implies \\
    p \equiv 1 \bmod 3}}
    \alpha_d \beta_e \chi_d(e)^s \chi_e(d)^t
    \bigg |
    \ll
    (A^{-\kappa_1} + B^{-\kappa_1})AB(\log AB)^{\kappa_2 }
    \]
    where $\alpha_d$, $\beta_e$ are non-zero complex numbers with absolute value bounded by one, that may depend on $d$ and $e$ respectively.
\end{lemma}
\begin{remark}
    Intuitively the values of non trivial characters seem to oscillate randomly as one varies the argument, thus summing a wide range of characters over a wide range of arguments (that are not chosen in an overly picky manner) should lead to some cancellation. The above lemma is a result quantifying this intuition for the special case that we need.
\end{remark}
\begin{proof}[Proof of Lemma \ref{DoubleOscillation}]
    As remarked before this result derives from Proposition \ref{AbstractLargeSieve}. We replace the Dirichlet characters by the respective power residue symbols.
    \[
    \chi_d(e)^s \chi_e(d)^t
    =
    \left(\frac{e}{\pi_d}\right)_{3}^s
    \left(\frac{d}{\pi_e}\right)_{3}^t
    =
    \left(\frac{N(\pi_e)}{\pi_d}\right)_{3}^s
    \left(\frac{N(\pi_d)}{\pi_e}\right)_{3}^t
    \]
    Here $N(\cdot)$ denotes the norm in $\Q[\zeta_3]$. We now choose the correct setup to apply Proposition \ref{AbstractLargeSieve}. Let $K = \Q[\zeta_3]$ and $\mathcal{O}_K = \Z[\zeta_3]$ the ring of integers. Recall that $\Z[\zeta_3]$ is a PID, in which rational primes equivalent to 1 modulo 3 split, rational primes equivalent to 2 modulo 3 are inert and $3 = (1-\zeta_3)(1-\zeta_3^2)$. As a consequence $N(z) \equiv 1$ mod 3 for all primes except $1-\zeta_3$, i.e. the 3rd power residue symbol is defined on all elements coprime to 3. Furthermore the ring of integers is a lattice with 6 units, $\{\pm1, \pm \zeta_3, \pm \zeta_3^2\}$. Let $l = 3$, $M = 27$, let $A_{\text{bad}}$ be the set of squarefull integers and let $S_{27\mathcal{O}_K}$ denote the set of elements in $\mathcal{O}_K$ that are coprime to $M\mathcal{O}_K = 27\mathcal{O}_K$. In other words $S_{27\mathcal{O}_K}$ is the set of integers coprime to 3. Define the following function
    \begin{gather*} \allowdisplaybreaks
        \gamma : S_{27\mathcal{O}_K} \times S_{27\mathcal{O}_K} \longrightarrow \{0, \zeta_3, \zeta_3^2, 1\} \\
        \gamma(w,z) = 
        \left(\frac{N(z)}{w}\right)_{3}^s
        \left(\frac{N(w)}{z}\right)_{3}^t.
    \end{gather*}
    We check that $\gamma$, together with the chosen parameters, satisfies the requirements in \ref{AbstractLargeSieve}. 
    
    P1) Multiplicativity: $\gamma(w_1w_2,z) = \gamma(w_1,z)\gamma(w_2,z)$ and $\gamma(w,z_1z_2)=\gamma(w,z_1)\gamma(w,z_2)$ follow from the multiplicativity of the norm and the multiplicativity of the power residue symbol in both its arguments.
    
    P2) Periodicity: For $z_1 \equiv z_2$ mod $N(w)$ and $z_1 \equiv z_2$ mod $27$ we also have $z_1 \equiv z_2$ mod $w$, so 
    \begin{align*} \allowdisplaybreaks
        \left(\frac{N(z_1)}{w}\right)_{3}
        \left(\frac{N(w)}{z_1}\right)_{3}
        &=
        \left(\frac{N(z_1)}{w}\right)_{3}
        \left(\frac{uz_1}{N(w)}\right)_{3}
        \allowdisplaybreaks\\
        &=
        \left(\frac{N(z_2)}{w}\right)_{3}
        \left(\frac{uz_2}{N(w)}\right)_{3}
        =
        \left(\frac{N(z_2)}{w}\right)_{3}
        \left(\frac{N(w)}{z_2}\right)_{3}
    \end{align*}
    and thus $\gamma(w,z_1) = \gamma(w,z_2)$. We again used multiplicativity of the norm and power residue symbol but now additionally cubic reciprocity. To see why cubic reciprocity may be applied, note that elements coprime to three may be multiplied by units to be primary. Thus in the above chain of equalities $u$ is some unit such that $uz_1$ is primary. Additionally, the power residue symbol $\left(\frac{N(w)}{z_1}\right)_{3}$ depends not on $z_1$, but the ideal generated by $z_1$, and the ideal does not change if the generator is multiplied by a unit.

    Furthermore, if $N(w) \notin A_{\text{bad}}$, then we have 
    \[
    \sum_{\substack{
    \xi \bmod M N(w)\\
    \gcd(\xi,M) = (1)
    }}
    \gamma(w,\xi)
    =
    \sum_{\substack{
    \xi \bmod 27 N(w)\\
    \gcd(\xi,27) = (1)
    }}
    \left(\frac{N(\xi)}{w}\right)_{3}^s
    \left(\frac{N(w)}{\xi}\right)_{3}^t
    =
    0
    \]
    The latter equality follows immediately from the orthogonality of characters if we can show that 
    \begin{align}
        \label{eChar}
    \xi 
    \mapsto 
    \left(\frac{N(\xi)}{w}\right)_{3}^s 
    \left(\frac{N(w)}{\xi}\right)_{3}^t
    =
    \left(\frac{\xi\bar{\xi}}{w}\right)_{3}^s 
    \left(\frac{w\bar{w}}{\xi}\right)_{3}^t
    \end{align} 
    is a nontrivial character on the finite abelian group $\mathcal{O}_K/27N(w)\mathcal{O}_K$. Clearly this is a character (see property P1), what remains is non-triviality. We construct an element $z \in \mathcal{O}_K$ such that the character does not evaluate to 1 at $z$.
    
    A calculation using cubic reciprocity for units and the prime $1-\zeta_3$ shows that if $z \equiv 1 \bmod 27$ and $w \in S_{27\mathcal{O}_K}$ we may use $\left(\frac{z}{w}\right)_{3} = \left(\frac{w}{z}\right)_{3}$ independently of the congruence class of $w \bmod 27$. Thus we restrict to $z \equiv 1 \bmod 27$, so we may freely apply cubic reciprocity. We have another useful relation for the third power residue symbol:
    $$
    \left(\left(\frac{x}{y}\right)_3\right)^2 = \overline{\left(\frac{x}{y}\right)}_3 = \left(\frac{\bar{x}}{\bar{y}}\right)_3.
    $$
    Using both relations, we rewrite equation \eqref{eChar} for $z \equiv 1 \bmod 27$ as
    $$
    z \mapsto \left(\frac{z}{w}\right)_3^{s + t} \left(\frac{\bar{z}}{w}\right)_3^{2s + t}.
    $$
    Since the linear map $\varphi: \mathbb{F}_3^2 \rightarrow \mathbb{F}_3^2$, given by $(s, t) \mapsto (s + t, 2s + t)$ is invertible, it follows that $(s + t, 2s + t)$ is also not congruent to $(0, 0)$ modulo $3$. Hence it suffices to show for every two integers $s', t'$, not both congruent to $0$ modulo $3$, that
    $$
    z \mapsto \left(\frac{z}{w}\right)_3^{s'} \left(\frac{\bar{z}}{w}\right)_3^{t'}.
    $$
    is not principal on the subgroup of integers $z \equiv 1 \bmod 27$.
    
    Recall that we are assuming $N(w) \notin A_{\text{bad}} = \{\text{squarefull integers}\}$, so let $p$ be a prime dividing $N(w)$ with multiplicity one. Then there exists a prime $\pi$ of $\Z[\zeta_3]$ such that $\pi$ divides $w$ but $\bar{\pi}$ does not. Write $g_1$ for an element of $\Z[\zeta_3]/\pi \Z[\zeta_3]$ with $\left(\frac{g_1}{\pi}\right)_3 = \zeta_3$ and write $g_2$ for an element of $\Z[\zeta_3]/\bar{\pi}\Z[\zeta_3]$ with $\left(\frac{g_2}{\bar{\pi}}\right)_3 = \zeta_3$. By the Chinese remainder theorem, we may find for every two integers $c, d$ some integer $z$, depending on $c$ and $d$, that satisfies
    \begin{gather*} \allowdisplaybreaks
    z \equiv 1 \bmod 27, \quad z \equiv 1 \bmod \frac{N(w)}{p} \\
    z \equiv g_1^c \bmod \pi, \quad z \equiv g_2^d \bmod \bar{\pi}.
    \end{gather*}
    Then the character, evaluated at this choice of $z$, becomes
    $$
    \left(\frac{z}{w}\right)_3^{s'} \left(\frac{\bar{z}}{w}\right)_3^{t'}
    =
    \left(\frac{z}{\pi}\right)_3^{s'} \left(\frac{\bar{z}}{\pi}\right)_3^{t'}
    \left(\frac{z}{w'}\right)_3^{s'} \left(\frac{\bar{z}}{w'}\right)_3^{t'}
    =
    \zeta_3^{s' c} \cdot \zeta_3^{2t' d} = \zeta_3^{s' c + 2t' d},
    $$
    where $w\coloneqq\pi w'$ and $\left(\frac{z}{w'}\right)_3^{s'} \left(\frac{\bar{z}}{w'}\right)_3^{t'} = 1$ because of the condition $z \equiv 1 \bmod \frac{N(w)}{p}$. For a good choice of $c, d$ the last expression is not $1$ and the proof is finished.
    
    P3) Bad count: $A_{\text{bad}}=\{$squarefull integers$\}$  satisfies the required bound with $C_2 = \frac{1}{2}$, if $C_1$ is sufficiently large (one could choose $C_1 = \zeta(\frac{3}{2})/\zeta(3)$, see \cite{Golomb})
    $$\sum_{\substack{
    n \in A_{\text{bad}}\\
    n \leq X}} 
    1 
    \leq C_1 X^{1 - C_2}
    \ll X^{\frac{1}{2}}.$$
    We have checked that $\gamma$ satisfies all requirements. We continue following the setup of Proposition \ref{AbstractLargeSieve}.
    
    All units $\mathcal{O}_K^{\ast} = \Z[\zeta_3]^{\ast} = \{\pm1, \pm \zeta_3, \pm \zeta_3^2\}$ in $K=\Q[\zeta_3]$ are torsion, hence in a decomposition $\Z[\zeta_3]^{\ast} = T \oplus V$ into torsion $T$ and free part $V$ (as in Dirichlet's unit theorem), $V$ is the trivial group. Consequently there is only one choice for a fundamental domain $\mathcal{D} \subseteq \mathcal{O}_K$ for the action of $V$ on $\mathcal{O}_K$, that is $\mathcal{D}= \mathcal{O}_K$.

    We are ready to rewrite the sum and apply Proposition \ref{AbstractLargeSieve}
    \begin{align*} \allowdisplaybreaks
    &\bigg | \sum_{
    \substack{
    1 < d \leq A \\
    p|d \implies \\
    p \equiv 1 \bmod 3}}
    \sum_{
    \substack{
    1 < e \leq B \\
    p|e \implies \\
    p \equiv 1 \bmod 3}}
    \alpha_d \beta_e \chi_d(e)^s \chi_e(d)^t
    \bigg |
    =
    \bigg | \sum_{
    \substack{
    w \in \mathcal{O}_K\\
    1 < N(w) \leq A}}
    \sum_{
    \substack{
    z \in \mathcal{O}_K\\
    1 < N(z) \leq B}}
    \alpha_{w}' \beta_{z}' \left(\frac{N(z)}{w}\right)_{3}^s \left(\frac{N(w)}{z}\right)_{3}^t
    \bigg |
    \allowdisplaybreaks\\
    =
    &\bigg | \sum_{
    \substack{
    w \in \mathcal{O}_K\\
    1 < N(w) \leq A}}
    \sum_{
    \substack{
    z \in \mathcal{O}_K\\
    1 < N(z) \leq B}}
    \alpha_{w}' \beta_{z}' \gamma(w,z)
    \bigg |
    \leq
    \sums{
    \delta,\epsilon \bmod 27\\
    \text{invertible} 
    }
    \bigg | \sum_{
    \substack{
    w \in \mathcal{O}_K\\
    1 < N(w) \leq A\\
    w \equiv \delta \bmod 27}}
    \sum_{
    \substack{
    z \in \mathcal{O}_K\\
    1 < N(z) \leq B\\
    z \equiv \epsilon \bmod 27}}
    \alpha_{w}' \beta_{z}' \gamma(w,z)
    \bigg |
    \allowdisplaybreaks\\
    =
    &
    \sums{
    \delta,\epsilon \bmod 27\\
    \text{invertible} 
    }
    |\mathcal{B}(A,B, \delta, \epsilon, 1, 1)|
    \ll
    (A^{-\frac{1}{12}}+B^{-\frac{1}{12}})AB\log(AB)^{\kappa_2 }.
    \end{align*}
    Here $\kappa_2 > 0$ is a constant and $\alpha_w', \beta_z'$ are complex numbers bounded by one, accounting for $\alpha_d, \beta_e$ from before and ensuring that the only summands appearing are the same ones as in the original sum. More precisely, $\alpha_w'$ and $\beta_z'$ also include the indicator functions of ``the rational prime factors of $N(w)$ and $N(z)$ are all equivalent to 1 mod 3" and ``the imaginary parts of all prime factors of $w$ and $z$ are greater or equal to 0".
\end{proof}
The second oscillation result is a ``single oscillation result". It derives from a theorem on sums of multiplicative functions from Koukoulopoulos' book \cite[Thm. 13.2, p.134]{KOUKOU}. It is included in the tools section of this paper, see Theorem \ref{Theorem 13.2 KOUKOU}.
\begin{lemma}
Let $i \in \{1,2\}$, let $m > 2$, $B > 0$ and $M \geq 1$. Let $0 < d_2, d_3 \leq (\log B)^M$ be coprime, square-free integers, that are only divisible by primes equivalent to 1 mod 3 and not both equal to 1. Then for all $C > 0$
    \label{SingleOscillation}
    $$
    \bigg|
    \sum_{
    \substack{
    B^{\sfrac{1}{m}} < e \leq B \\
    p \mid e \implies p \equiv i \bmod 3
    }}
    \mu^2(e)\frac{\chi_{d_2}(e)\chi_{d_3}^2(e)}{3^{\omega(e)}}
    \bigg|
    \ll_{C,m,M} 
    \frac{B}{(\log B)^{C}}.
    $$
    Furthermore we obtain the same bound if we omit the factor $\frac{1}{3^{\omega(e)}}$, 
    $$
    \bigg|
    \sum_{
    \substack{
    B^{\sfrac{1}{m}} < e \leq B \\
    p \mid e \implies p \equiv i \bmod 3
    }}
    \mu^2(e)\chi_{d_2}(e)\chi_{d_3}^2(e)
    \bigg|
    \ll_{C,m,M} 
    \frac{B}{(\log B)^{C}}.
    $$
\end{lemma}
\begin{proof}[Proof of Lemma \ref{SingleOscillation}]
    We show how to bound the sum
    \begin{gather*}
    \sum_{
    \substack{
    B^{\sfrac{1}{m}} < e \leq B \\
    p \mid e \implies p \equiv 1 \bmod 3
    }}
    \mu^2(e)\frac{\chi_{d_2}(e)\chi_{d_3}^2(e)}{3^{\omega(e)}},
    \end{gather*}
    but throughout the proof it will be clear that omitting the factor $\frac{1}{3^{\omega(e)}}$ or changing the summation condition ``$p \mid e \implies p \equiv 1 \bmod 3$" to ``$p \mid e \implies p \equiv 2 \bmod 3$" does not change the bound.
    
    As remarked before, the lemma derives from a result on sums of multiplicative functions, Theorem \ref{Theorem 13.2 KOUKOU}. We now choose the correct parameters to apply the theorem to our character sum. Let $\alpha = 0, k=1, Q = \exp((3\log B)^{\frac{1}{100}})$, $J$ some large integer to be specified later and $\epsilon > 0$ some real number to be specified later. Define the function
    $$f(n) = \mu^2(n)\frac{\chi_{d_2}(n)\chi_{d_3}^2(n)}{3^{\omega(n)}}\mathbb{1}_{\{p|n \implies p \equiv 1 \bmod 3\}}.$$
    Note that $f(n)$ is multiplicative and bounded in absolute value by one; in particular $|f(n)| \leq \tau_k(n) = \tau_1(n)$. We want to verify that $f(n)$ satisfies condition \eqref{KOKOU condition} of \ref{Theorem 13.2 KOUKOU}, i.e. we wish to show that for all $A' > 0$ and $x \geq Q$,
    $$
    \sum_{p \leq x} f(p)\log(p) = \mathcal{O}_{A'}\left(\frac{x}{(\log x)^{A'}}\right).
    $$
    So let $x \geq Q$ and $A'>0$. Then
    \begin{align*}
    \allowdisplaybreaks
        &\sum_{p \leq x} f(p)\log(p)
        = 
        \sums{
        p \leq x\\
        p \equiv 1 \bmod 3} 
        \frac{\chi_{d_2}(p)\chi_{d_3}^2(p)}{3}\log p
        \allowdisplaybreaks\\=
        &\sums{
        b \in (\sfrac{\Z}{3d_2d_3\Z})^{\ast} \\
        b \equiv 1 \bmod 3
        }
        \sums{
        p \leq x \\
        p \equiv b \bmod 3d_2d_3
        }
        \frac{\chi_{d_2}(p)\chi_{d_3}^2(p)}{3}\log p
        \allowdisplaybreaks\\=
        &\sums{
        b \in (\sfrac{\Z}{3d_2d_3\Z})^{\ast} \\
        b \equiv 1 \bmod 3
        }
        \frac{\chi_{d_2}(b)\chi_{d_3}^2(b)}{3}
        \sums{
        p \leq x \\
        p \equiv b \bmod 3d_2d_3
        }
        \log p
        \allowdisplaybreaks\\ =
        &\sums{
        b \in (\sfrac{\Z}{3d_2d_3\Z})^{\ast} \\
        b \equiv 1 \bmod 3
        }
        \frac{\chi_{d_2}(b)\chi_{d_3}^2(b)}{3}
        \left(\frac{x}{\phi(3d_2d_3)} + \mathcal{O}\bigg(x \exp\big(-C_{N}(\log x)^{\frac{1}{2}}\big)\bigg)\right).
    \end{align*}
    We applied the Siegel--Walfisz theorem in the last step, defining $N=200M$ and using that $3d_2d_3 \leq (3\log B)^{2 M} = (\log Q)^{200M} \leq (\log x)^{N}$. Here $C_{N} > 0$ is some constant depending on $N$, coming from the Siegel--Walfisz theorem. Since $d_2, d_3$ are coprime and not both equal to 1, $\chi_{d_2}\chi_{d_3}^2$ is not the trivial character. By orthogonality of characters, we get
    \begin{gather*}
        \allowdisplaybreaks
        \sums{
    b \in (\sfrac{\Z}{3d_2d_3\Z})^{\ast} \\
    b \equiv 1 \bmod 3
    }
    \frac{\chi_{d_2}(b)\chi_{d_3}^2(b)}{3}
    =
    0,
    \end{gather*}
    thus
    \begin{align*}
        \allowdisplaybreaks
        \sum_{p \leq x} f(p)\log(p) 
        &= 
        \mathcal{O}\left(d_2d_3x \exp\big(-C_{N}(\log x)^{\frac{1}{2}}\big)\right)
        \allowdisplaybreaks\\
        &= 
        \mathcal{O}\left(x \exp\big(-C'_{N}(\log x)^{\frac{1}{2}}\big)\right)
        =
        \mathcal{O}\left(\frac{x}{(\log x)^{A'}}\right).
    \end{align*}
    
   Note that for $\epsilon = 1$, we have $B \geq \exp((\log Q)^{1+\epsilon}) = \exp((3 \log B)^{\frac{2}{100}})$, so we may apply Theorem \ref{Theorem 13.2 KOUKOU} to our sum as follows
    \begin{align*}
    \allowdisplaybreaks
    &\bigg|\sums{
    B^{\sfrac{1}{m}} < e \leq B \\
    p \mid e \implies p \equiv 1 \bmod 3
    }
    \mu^2(e)\frac{\chi_{d_2}(e)\chi_{d_3}^2(e)}{3^{\omega(e)}}\bigg|
    \allowdisplaybreaks\\ \leq
    &\bigg|\sums{
    1 \leq e \leq B \\
    p \mid e \implies p \equiv 1 \bmod 3
    }
    \mu^2(e)\frac{\chi_{d_2}(e)\chi_{d_3}^2(e)}{3^{\omega(e)}}\bigg|
    +
    \bigg|\sums{
    1 \leq e \leq B^{\sfrac{1}{m}} \\
    p \mid e \implies p \equiv 1 \bmod 3
    }
    \mu^2(e)\frac{\chi_{d_2}(e)\chi_{d_3}^2(e)}{3^{\omega(e)}}\bigg|
    \allowdisplaybreaks\\ \leq
    & \bigg|\sums{e \leq B}f(e)\bigg| + B^{\sfrac{1}{m}}
    =
    \OO\left(\frac{B(\log Q)^{J+1}}{(\log B)^{J+1}}\right)
    +
    B^{\sfrac{1}{m}}.
    \end{align*}
    In the last step we used Theorem \ref{Theorem 13.2 KOUKOU} with the parameters specified at the start of the proof and importantly followed the convention in Remark \ref{convention}. To end the proof, we check that the above bound is as strong as required in the lemma,
    \begin{align*}
        \allowdisplaybreaks
    &\OO\left(\frac{B(\log Q)^{J+1}}{(\log B)^{J+1}}\right)
    +
    B^{\sfrac{1}{m}}
    =
    \OO\left(\frac{B((3\log B)^\frac{J+1}{100})}{(\log B)^{J+1}}\right)
    +
    B^{\sfrac{1}{m}}
    \allowdisplaybreaks\\ &=
    \OO\left(\frac{B}{(\log B)^{\frac{99}{100}(J+1)}}\right)
    =
    \OO\left(\frac{B}{(\log B)^{C}}\right),
    \end{align*}
    where $J$ is chosen large enough in terms of $C$. Closer inspection of the proof and the statements of Siegel--Walfisz and Theorem \ref{Theorem 13.2 KOUKOU} reveals that the implied constant depends at most on $C$, $m$ and $M$.
\end{proof}
\subsection{Proof of Theorem \ref{AsymptoticFormula}}
    The larger idea of the proof is to rewrite $N(A,B)$ as a sum over cubic characters, split the sum into several ranges and then apply character summation results. The main term will come from the case where all characters are principal.
    
    We rewrite $N(A,B)$. Observe that Lemma \ref{p nmid ab} and Lemma \ref{p mid ab} imply that for ``admissible" pairs $(a,b)$ of coprime, square-free integers, not divisible by three, $\phi_{a,b}$ is locally soluble everywhere except possibly at primes that satisfy $p \mid ab$ and $p \equiv 1$ mod $3$. Furthermore Lemma \ref{p mid ab} provides an indicator function for ``$\phi_{a,b}$ has a solution in $\QP$" for such primes
    \begin{gather*} \allowdisplaybreaks
        N(A,B)
        = 
        \sum_{\substack{a \leq A, \; b \leq B \\ \gcd(ab,3) = 1, \; \gcd(a,b) = 1 \\ a,b \; \text{square-free} \\ (a,b) \; \text{admissible pair}}} \mathbb{1}_{\{\phi_{a,b} \, \text{is locally soluble everywhere}\}} \\
        = 
        \sum_{\substack{a \leq A, \; b \leq B \\
        (a,b) \text{ admissible pair}}} \;
        \mu^2(3ab)\prod_{\substack{p \mid a \\ p \equiv 1 \bmod 3}} \frac{1}{3}(1+\chi_p(b)+\chi_p^2(b))
        \prod_{\substack{p \mid b \\ p \equiv 1 \bmod 3}} \frac{1}{3}(1+\chi_p(a)+\chi_p^2(a)).
    \end{gather*}
    We collect the prime factors of $a$ and $b$ into new variables, according to their residue class mod 3. Define $a = a_1a_2$ and $b= b_1b_2$, where
    \[
    \begin{cases}
        p \mid a_1b_1 \implies p \equiv 1 \bmod 3, \\
        p \mid a_2b_2 \implies p \equiv 2 \bmod 3. \\
    \end{cases}
    \]
    For improved readability collect the summation conditions into the ``$\flat$" symbol:
    \[
    \sum_{
     \substack{
     a_1a_2 \leq A, \; b_1b_2 \leq B \\
    (a_1a_2,b_1b_2) \, \text{is an admissible pair}\\
    p \mid a_ib_i \implies p \equiv i \bmod 3 \; \text{for} \; i \in \{1,2\}
    }} 
    \coloneqq 
    \sum_{
    \substack{a_1a_2 \leq A \\
        b_1b_2 \leq B \\
        \flat}}.
    \]
    With the new notation and variable names we have
    \begin{gather} \allowdisplaybreaks
    \begin{align*} \allowdisplaybreaks
        N(A,B) 
        &= 
        \sum_{\substack{a_1a_2 \leq A \\
        b_1b_2 \leq B \\ 
        \flat}} \;
        \frac{\mu^2(3a_1a_2b_1b_2)}{3^{\omega(a_1b_1)}}
        \prod_{p \mid a_1} \left(1+\chi_p(b_1b_2)+\chi_p^2(b_1b_2)\right) 
        \prod_{p \mid b_1} \left(1+\chi_p(a_1a_2)+\chi_p^2(a_1a_2)\right) \\
        &= 
        \sum_{\substack{a_1a_2 \leq A \\
        b_1b_2 \leq B \\ 
        \flat}} \;
        \frac{\mu^2(3a_1a_2b_1b_2)}{3^{\omega(a_1b_1)}}
        \sum_{d_1d_2d_3 = a_1} \chi_{d_2}(b_1b_2)\chi_{d_3}^2(b_1b_2) \sum_{e_1e_2e_3 = b_1} \chi_{e_2}(a_1a_2)\chi_{e_3}^2(a_1a_2)
    \end{align*}
    \\
    =   \sum_{
        \substack{
        d_1d_2d_3a_2 \leq A \\
        e_1e_2e_3b_2 \leq B \\
        \flat}
        }
        \frac{\mu^2(3d_1d_2d_3a_2e_1e_2e_3b_2)}{3^{\omega(d_1d_2d_3e_1e_2e_3)}} 
        \chi_{d_2}(e_1e_2e_3b_2)
        \chi_{d_3}^2(e_1e_2e_3b_2)
        \chi_{e_2}(d_1d_2d_3a_2)
        \chi_{e_3}^2(d_1d_2d_3a_2).
        \label{BigCharacterSum}
    \end{gather}
    Note that we implicitly also replaced $a_1$ with $d_1d_2d_3$ and $b_1$ with $e_1e_2e_3$ in the summation conditions $\flat$.
    
    The ``big character sum" \eqref{BigCharacterSum} is of the right type to apply oscillation results. In what follows we will call the summation variables $d_2, d_3, e_2, e_3$ ``indices", because they appear in the indices of the characters $\chi_{d_2}, \chi_{d_3}, \chi_{e_2}, \chi_{e_3}$. Throughout the proof we use a parameter $V$ (that will be chosen as a large power of $(\log AB)$ at the end of the proof) to partition the range of summation in \eqref{BigCharacterSum}. There are four cases: With the first and second case, we deal with all summation ranges where at least one of the indices is larger than $V$. The third case covers the ranges where all indices are smaller than $V$, but not all characters are trivial. In the fourth case all indices are equal to one, i.e. all characters are trivial, and we obtain the main term of the theorem. All cases except the fourth are bounded such that they fit into the error term.
    
    For the first case assume that at least one index (one of the variables $d_2,d_3,e_2,e_3$) in \eqref{BigCharacterSum} is larger than $V$, say $d_2 > V$, and that at least one of the variables inside the argument of the respective character is larger than $V$, say $e_1 >V$. To clarify, there are 12 ``sub-cases" of case one, i.e. 12 ranges of summation that may all be bounded in the same way. We choose to show the argument for the ``sub-case" $d_2 > V, e_1>V$. We apply the triangle inequality to isolate $\chi_{d_2}$ and use Lemma \ref{DoubleOscillation} to obtain the following contribution to $N(A,B)$,
    \begin{align*} 
    \allowdisplaybreaks
    &\big|N(A,B)_{d_2>V, e_1>V}\big|
        \leq 
        \sum_{
        \substack{
        d_1d_3a_2 \leq A\\
        e_2e_3b_2 \leq B\\
        \flat
        }}
        \bigg|
        \sum_{
        \substack{
        V < d_2 \leq \frac{A}{d_1d_3a_2}\\
        V < e_1 \leq \frac{B}{e_2e_3b_2}\\
        \flat}
        }
        \frac{\mu^2(3d_1d_2d_3a_2e_1e_2e_3b_2)}{3^{\omega(d_1d_2d_3e_1e_2e_3)}}
        \chi_{d_2}(e_1e_2e_3b_2)
        \bigg|
        \allowdisplaybreaks\\
        &\leq
        \sum_{
        \substack{
        d_1d_3a_2 \leq A\\
        e_2e_3b_2 \leq B\\
        \flat
        }}
        \frac{\mu^2(d_1d_3a_2e_2e_3b_2)}{3^{\omega(d_1d_3e_2e_3)}}
        \bigg|
        \sum_{
        \substack{
        V < d_2 \leq \frac{A}{d_1d_3a_2}\\
        V < e_1 \leq \frac{B}{e_2e_3b_2}\\
        \flat
        }}
        \alpha_{d_2}\beta_{e_1}\chi_{d_2}(e_1)
        \bigg|
        \allowdisplaybreaks\\
        &\ll
        \sum_{
        \substack{
        d_1d_3a_2 \leq A\\
        e_2e_3b_2 \leq B\\
        \flat
        }}
        \frac{\mu^2(d_1d_3a_2e_2e_3b_2)}{3^{\omega(d_1d_3e_2e_3)}}
        \frac{A}{d_1d_3a_2}\frac{B}{e_1e_3b_2}
        \left(\bigg(\frac{A}{d_1d_3a_2}\bigg)^{-\kappa_1} + \bigg(\frac{B}{e_1e_3b_2}\bigg)^{-\kappa_1}\right)
        \log(AB)^{\kappa_2 }
        \allowdisplaybreaks\\
        &\ll
        \frac{AB\log(AB)^{\kappa_2 }}{V^{\kappa_1}}
        \sum_{
        \substack{
        d_1d_3a_2 \leq A\\
        e_2e_3b_2 \leq B\\
        \flat
        }}
        \frac{1}{3^{\omega(d_1d_3e_2e_3)}}
        \frac{\mu^2(d_1d_3a_2e_2e_3b_2)}{d_1d_3a_2e_1e_3b_2},
        \end{align*}
    where $\alpha_{d_2}$ and $\beta_{e_1}$ are complex numbers with absolute values $\leq 1$, absorbing the summation condition ``admissible pair", the factor $\frac{1}{3^{\omega(d_2e_1)}}\chi_{d_2}(e_2e_3b_2)$ and the fact that $e_1,d_2$ are square-free.
    
    To evaluate the last sum, remember that the remaining summation conditions are $p\mid d_1d_3e_2e_3 \implies p \equiv 1$ mod 3 and $p \mid a_2b_2 \implies p\equiv 2$ mod 3. We make two variable substitutions, $m \coloneqq d_1d_3e_2e_3$ and $n \coloneqq a_2b_2$,
    \begin{gather*} \allowdisplaybreaks 
        \sum_{
        \substack{
        d_1d_3a_2 \leq A\\
        e_2e_3b_2 \leq B\\
        \flat
        }}
        \frac{1}{3^{\omega(d_1d_3e_2e_3)}}
        \frac{\mu^2(d_1d_3a_2e_2e_3b_2)}{d_1d_3a_2e_1e_3b_2}
        \leq
        \sum_{
        \substack{
        m \leq AB\\
        n \leq AB\\
        \flat
        }}
        \frac{4^{\omega(m)}2^{\omega(n)}}{3^{\omega(m)}}
        \frac{\mu^2(mn)}{mn}
        \allowdisplaybreaks\\
        \leq
        \sum_{
        \substack{
        m \leq AB\\
        p \mid m \implies p \equiv 1 \bmod 3
        }}
        \left(\frac{4}{3}\right)^{\omega(m)}
        \frac{\mu^2(m)}{m}
        \sum_{
        \substack{
        n \leq AB \\
        p \mid n \implies p \equiv 2 \bmod 3
        }
        }
        \frac{\mu^2(n)}{n}
        \ll
        (\log AB)^{\frac{4}{3}}
        (\log AB),
    \end{gather*}
    where in the last step we used that $f(m) \coloneqq (\frac{4}{3})^{\omega(m)}\mu^2(m)$ is a multiplicative function, thus $\sum_{m\leq AB} \frac{f(m)}{m} 
    \leq \prod_{p \leq AB}(1+\frac{f(p)}{p} + \frac{f(p^2)}{p^2} + ... )
    = \prod_{p \leq AB}(1+\frac{4}{3p}) 
    \ll \prod_{p \leq AB}(1+\frac{1}{p})^{\frac{4}{3}} 
    \ll (\log AB)^{\frac{4}{3}}$. Overall for the sub-case $d_2 > V$, $e_1 > V$ and all analogous sub-cases we get a contribution to $N(A,B)$ of 
    \begin{equation}
    \label{Case1}
    \OO\left(\frac{AB(\log AB)^{\kappa_2 + \frac{7}{3}}}{V^{\kappa_1}}\right).
    \end{equation}
    
    For the second case, assume that one of the indices in \eqref{BigCharacterSum} is larger than $V$ but none of the variables in the argument of that character are larger than $V$. We will show the sub-case that $d_2>V$ and $e_1,e_2,e_3,b_2 \leq V$, but our argument works the same for the other sub-cases. We use the triangle inequality and bound trivially
    \begin{align*} \allowdisplaybreaks
        \big|
        N(A,B)_{\substack{d_2 > V \\ 
        e_1,e_2,e_3,b_2 \leq V}}
        \big|
        &\leq
        \sum_{
        \substack{
        e_1e_2e_3b_2 \leq B \\
        e_1,e_2,e_3,b_2 \leq V \\
        \flat
        }
        }
        \sum_{
        \substack{
        d_1d_2d_3a_2 \leq A \\
        \flat
        }
        }
        \frac{\mu^2(3d_1d_2d_3a_2e_1e_2e_3b_2)}{3^{\omega(d_1d_2d_3e_1e_2e_3)}}
        \allowdisplaybreaks\\
        &\leq
        V^4
        \sum_{
        \substack{
        d_1d_2d_3a_2 \leq A \\
        p \mid d_1d_2d_3 \implies p \equiv 1\bmod 3 \\
        p \mid a_2 \implies p \equiv 2\bmod 3
        }
        }
        \frac{\mu^2(d_1d_2d_3a_2)}{3^{\omega(d_1d_2d_3)}}
        \allowdisplaybreaks
        \\
        &=
        V^4
        \sum_{
        \substack{
        n \leq A
        }
        }
        \mu^2(n)
        \allowdisplaybreaks\\
        &\leq
        V^4A.
    \end{align*}
    For all four analogous sub-cases where one of the indices is large, but none of the respective arguments, we get a contribution to $N(A,B)$ of 
    \begin{equation}
    \label{Case2}
    \OO\left(V^4A+V^4B\right).
    \end{equation}
    
    With the previous two cases we have fully dealt with the summation ranges where at least one of the indices in the big character sum \eqref{BigCharacterSum} is large. We now turn to the summation range where all indices are small, i.e. $d_2,d_3,e_2,e_3 \leq V$.

    For the third case assume that $d_2,d_3,e_2,e_3 \leq V$, but not $d_2 = d_3 = e_2 = e_3 = 1$, i.e. that at least one index is larger than one. We choose $1 < d_2 \leq V$, but there are four equivalent sub-cases. We want to apply a single oscillation result \ref{SingleOscillation} to the character $\chi_{d_2}(e_1e_2e_3b_2)$. More specifically we would like to apply it to a sum of the form $\sum_{e_1 \leq \frac{B}{e_2e_3b_2}} \frac{\chi_{d_2}(e_1)}{3^{\omega(e_1)}}$ to obtain a $\log(\frac{B}{e_2e_3b_2})^{-C}$ saving compared to the trivial bound. If no further restrictions are placed on $e_1$, the $\log(\frac{B}{e_2e_3b_2})^{-C}$ saving becomes poor when $e_2e_3b_2$ is close to $B$. Thus we need to divide the third case into two further ranges: one where $e_1,b_2 \leq B^{\sfrac{1}{k}}$ (call this case 3a) and one where $e_1 > B^{\sfrac{1}{k}}$ or $b_2 > B^{\sfrac{1}{k}}$ (call this case 3b). Here $k$ is some integer greater than 2 that can be chosen later for convenience ($k=4$ works for example).

    So we turn to case 3a, i.e. we assume $d_2,d_3,e_2,e_3 \leq V$, $d_2 > 1$, $b_2,e_1 \leq B^{\sfrac{1}{k}}$. Then we may bound trivially, similarly to case 2:
    \begin{align*} \allowdisplaybreaks
    \big|N(A,B)_{\substack{
    d_2,d_3,e_2,e_3 \leq V \\
    d_2 > 1 \\
    e_1,b_2 \leq B^{\sfrac{1}{k}}
    }}\big|
    &\leq
    \sum_{\substack{
    e_1e_2e_3b_2 \leq B \\
    e_1,b_2 \leq B^{\sfrac{1}{k}}\\
    e_2,e_3 \leq V \\
    \flat
    }}
    \sum_{\substack{
    d_1d_2d_3a_2 \leq A \\
    \flat
    }}
    \frac{\mu^2(3d_1d_2d_3a_2e_1e_2e_3b_2)}{3^{\omega(d_1d_2d_3e_1e_2e_3)}}
    \allowdisplaybreaks\\
    &\leq
    B^{\sfrac{2}{k}}V^2
    \sum_{\substack{
    d_1d_2d_3a_2 \leq A \\
    p \mid d_1d_2d_3 \implies p \equiv 1 \bmod 3 \\
    p \mid a_2 \implies p \equiv 2 \bmod 3
    }}
    \frac{\mu^2(d_1d_2d_3a_2)}{3^{\omega(d_1d_2d_3)}}
    \allowdisplaybreaks\\
    &\leq
    B^{\sfrac{2}{k}}V^2
    \sum_{
    n \leq A
    }
    \mu^2(n)
    \allowdisplaybreaks\\
    &\leq
    B^{\sfrac{2}{k}}V^2A.
    \end{align*}
    For all cases analogous to 3a, we get a contribution to $N(A,B)$ of
    \begin{equation}
    \label{Case3a}
    \OO\left(ABV^2\big(A^{\frac{2-k}{k}} + B^{\frac{2-k}{k}}\big)\right).
    \end{equation}
    
    We turn to case 3b, i.e. we assume $d_2,d_3,e_2,e_3 \leq V$, $d_2 > 1$ and that at least one of $b_2, e_1$ is larger than $B^{\sfrac{1}{k}}$. We choose to show the argument for $e_1 > B^{\sfrac{1}{k}}$, but Lemma \ref{SingleOscillation} is flexible enough to deal with the slightly different sub-case $b_2 > B^{\sfrac{1}{k}}$ in just the same way. Let $C>0$ be a parameter to be chosen later. Then we have,
    \begin{align*} \allowdisplaybreaks
        &\big|
        N(A,B)_{\substack{
        d_2,d_3,e_2,e_3 \leq V \\ 
        d_2 > 1\\
        e_1 > B^{\sfrac{1}{k}}}}
        \big|
        \\
        &\leq
        \sum_{
        \substack{
        d_1d_2d_3a_2 \leq A \\
        e_2e_3b_2 \leq B \\
        d_3,e_2,e_3 \leq V \\
        1 < d_2 \leq V \\
        \flat
        }
        }
        \frac{\mu^2(3d_1d_2d_3a_2e_2e_3b_2)}{3^{\omega(d_1d_2d_3e_2e_3)}} \;
        \bigg|
        \sum_{
        \substack{
        B^{\sfrac{1}{k}} < e_1 \leq \frac{B}{e_2e_3b_2} \\
        p \mid e_1 \implies p \equiv 1 \bmod 3
        }
        }
        \mu^2(e_1)\frac{\chi_{d_2}(e_1)\chi_{d_3}^2(e_1)}{3^{\omega(e_1)}}
        \bigg|
        \allowdisplaybreaks\\
        &\ll_C 
        \sum_{\substack{
        d_1d_2d_3a_2 \leq A \\
        \flat
        }}
        \sum_{
        \substack{
        e_2e_3b_2 \leq B \\
        \flat
        }
        }
        \frac{\mu^2(3d_1d_2d_3a_2e_2e_3b_2)}{3^{\omega(d_1d_2d_3)}3^{\omega(e_2e_3)}}
        \frac{B}{e_2e_3b_2}(\log(B^{\sfrac{1}{k}}))^{-C}
        \allowdisplaybreaks\\
        &\ll_C
        \frac{B}{(\log B)^C}
        \sum_{\substack{
        d_1d_2d_3a_2 \leq A \\
        \flat
        }}
        \frac{\mu^2(3d_1d_2d_3a_2)}{3^{\omega(d_1d_2d_3)}}
        \sum_{
        \substack{
        e_2e_3b_2 \leq B \\
        \flat
        }}
        \frac{\mu^2(3e_2e_3b_2)}{3^{\omega(e_2e_3)}}
        \frac{1}{e_2e_3b_2}
        \allowdisplaybreaks\\
        &\ll_C
        \frac{B}{(\log B)^C}
        \sum_{\substack{
        m \leq A
        }}
        \sum_{
        \substack{
        n \leq B
        }}
        \frac{1}{n}
        \allowdisplaybreaks\\
        &\ll_C
        \frac{AB}{(\log B)^{C-1}}.
    \end{align*}
    Here we applied Lemma \ref{SingleOscillation} to evaluate the character sum in the second line and used that $B^{\sfrac{1}{k}} < \frac{B}{e_2e_3b_2}$, hence $\log(B^{\sfrac{1}{k}})^{-C} > \log(\frac{B}{e_2e_3b_2})^{-C}$. Note that in the second line, $d_2,d_3$ are bounded by a large power of $\log B$ as required in Lemma \ref{SingleOscillation} because of the assumption that $A,B \geq \exp ((\log AB)^{\delta})$ and the fact that $V$ is a power of $\log AB$ (yet to be chosen). For all sub-cases analogous to 3b we get a contribution to $N(A,B)$ of
    \begin{gather}
    \label{Case3b}
    \OO\left(ABV^2\bigg(\frac{1}{(\log A)^{C-1}}+\frac{1}{(\log B)^{C-1}}\bigg)\right).
    \end{gather}
    
    This leaves us with the fourth and last range of summation, the range where $d_2 = d_3 = e_2 = e_3 = 1$, i.e. all characters in the big character sum \eqref{BigCharacterSum} are trivial. This range constitutes the main contribution to $N(A,B)$.
    
    For the fourth case, assume $d_2 = d_3 = e_2 = e_3 = 1$. Then the big character sum simplifies to
    \begin{align} \allowdisplaybreaks
    \nonumber
    N(A,B)_{\substack{
    d_2=d_3= 1 \\
    e_2=e_3=1
    }}
    =
    \sum_{\substack{
    d_1a_2 \leq A \\
    e_1b_2 \leq B \\
    \flat
    }} 
    \frac{\mu^2(3d_1a_2e_1b_2)}{3^{\omega(d_1e_1)}}
    =&
    \sum_{\substack{
    m \leq A
    }}
    \frac{\mu^2(3m)}{3^{\omega_1(m)}}
    \sum_{\substack{
    n \leq B \\
    (m,n) \, \text{admissible}
    }}
    \frac{\mu^2(3mn)}{3^{\omega_1(n)}}
    \allowdisplaybreaks\\
    =&
    \label{trivialwithadmissible}
    \sum_{\substack{
    m \leq A
    }}
    g(m)
    \sum_{\substack{
    n \leq B \\
    (m,n) \, \text{admissible}
    }}
    g_m(n).
    \end{align}
    Here we defined $\omega_1(n)$ to be the function that counts the number of prime divisors of $n$ which are congruent to 1 mod 3, $\omega_1(n) \coloneqq \#\{p \text{ prime} : p \equiv 1 \text{ mod 3 and } p \mid n \}$. Furthermore we define the functions
    $$
    g_m(n) \coloneqq \frac{\mu^2(3mn)}{3^{\omega_1(n)}} \quad \text{and} \quad g(m) \coloneqq \frac{\mu^2(3m)}{3^{\omega_1(m)}}
    $$
    and note that they are multiplicative. In order to evaluate the sum \eqref{trivialwithadmissible}, we want to apply a theorem on sums of multiplicative functions, Theorem \ref{Theorem 13.2 KOUKOU}. But in order to apply the theorem we first have to deal with the summation condition ``$(m,n)$ admissible", which is not multiplicative.
    
    Denote by $G = (\sfrac{\Z}{27\Z})^{\ast}$ the multiplicative group modulo 27 and by $\widehat{G} = \widehat{(\sfrac{\Z}{27\Z})^{\ast}}$ the group of characters modulo 27. Using orthogonality of characters, we may express the indicator function of a pair $(m,n)$ being admissible as follows,
    \begin{gather*}
        \mathbb{1}_{\{(m,n) \text{ admissible}\}} 
        = 
        \frac{1}{\phi(27)^2}
        \sums{\delta, \epsilon \bmod 27\\
        (\delta, \epsilon) \text{ admissible}}
        \sum_{\chi, \chi' \in G}
        \chi(\delta^{-1}m)\chi'(\epsilon^{-1}n).
    \end{gather*}
    We use this expression to rewrite \eqref{trivialwithadmissible}
    \begin{align}
    \allowdisplaybreaks
    \nonumber
    &N(A,B)_{\substack{
    d_2=d_3= 1 \\
    e_2=e_3=1
    }}
    =
    \sum_{\substack{
    m \leq A
    }}
    g(m)
    \sum_{\substack{
    n \leq B \\
    (m,n) \, \text{admissible}
    }}
    g_m(n)
    \allowdisplaybreaks\\
    \label{Admissiblepulledout}
    &=
    \frac{1}{\phi(27)^2}
    \sums{\delta, \epsilon \bmod 27\\
    (\delta, \epsilon) \text{ admissible}}
    \sum_{\chi, \chi' \in G}
    \chi(\delta^{-1})\chi'(\epsilon^{-1})
    \sum_{
    m \leq A
    }
    \chi(m)g(m)
    \sum_{
    n \leq B
    }
    \chi'(n)g_m(n).
    \end{align}
    The sum $\sum_{n \leq B}\chi'(n)g_m(n)$ is now of the right shape for an application of Theorem \ref{Theorem 13.2 KOUKOU}, so we check that the multiplicative function $\chi'(n)g_m(n)$ satisfies the requirement \eqref{KOKOU condition} for any choice of $\chi'$ and $m$. There are three possibilities we treat separately;
    \begin{enumerate}[i)]
    \allowdisplaybreaks
        \item $\chi'$ is the trivial character,
        \item $\chi'$ is non-trivial on $G$ but trivial on the index two subgroup $H = \{z \in G : z \equiv 1\bmod 3\}$,
        \item $\chi'$ is non-trivial on $H$.
    \end{enumerate}
    
    i) If $\chi'$ is the trivial character, we show that $\chi'(n)g_m(n) = g_m(n)$ satisfies condition \eqref{KOKOU condition} with parameters $\alpha = \sfrac{2}{3}$, $k=1$, $J = 2$, $Q = \exp((\log B)^{\frac{1}{100}})$, $\epsilon = 1$. For all $x \geq Q$ and all $A' > 0$,
    \begin{align*}
    \allowdisplaybreaks
        \sum_{p \leq x}\chi'(p)g_m(p) \log p 
        &= 
        \sums{
        p \leq x \\
        p \equiv 1 \bmod 3} 
        \frac{1}{3} \log p 
        + 
        \sums{
        p \leq x \\
        p \equiv 2 \bmod 3}
        \log p 
        - 
        \sums{p \mid m} \frac{1}{3^{\omega_1(p)}} \log p
        \allowdisplaybreaks\\
        &=
        \frac{1}{3}\cdot\frac{1}{\phi(3)}x + \frac{1}{\phi(3)}x + \OO(x \exp( - C_{W}(\log x)^{\sfrac{1}{2}})) - \OO(\omega(m)\log m)
        \allowdisplaybreaks\\
        &=
        \frac{2}{3}x + \OO\left(\frac{x}{(\log x)^{A'}}\right).
    \end{align*}
    In the second line we have applied the Siegel--Walfisz theorem and $C_{W}>0$ is some constant coming from this theorem. Since we will apply Siegel--Walfisz multiple times in the remainder of the proof and the associated constants will always just be absorbed into $\OO(\frac{x}{(\log x)^{A'}})$, we will abuse notation and call them $C_W$ every time. Note that we used that $\log m \leq \log A$ and the assumption that $A,B \geq \exp ((\log AB)^{\delta})$ to bound the second error term.

    In case that $\chi'$ is trivial, applying Theorem \ref{Theorem 13.2 KOUKOU} (with the parameters chosen above) to the sum $\sum
    \chi'(n)g_m(n)$ yields
    \begin{align}
    \label{gmtrivial}
    \sums{
        n \leq B
    }
    \chi'(n)g_m(n)
    =
    \frac{B}{(\log B)^{\frac{1}{3}}\Gamma(\frac{2}{3})}c_0(m) + \OO\bigg(\frac{B}{(\log B)^{\frac{4}{3}}}\bigg),
    \end{align}
    where
    $$
    c_0(m) = \prod_{p}\left(1+\frac{g(p)}{p}\right)\left(1- \frac{1}{p}\right)^{\frac{2}{3}}\prod_{p\mid m}\left(1+\frac{g(p)}{p}\right)^{-1}.
    $$
    
    ii) If $\chi'$ is non-trivial on $G$ but trivial on $H$, there is a small difference to case i) when showing that $\chi'(n)g_m(n)$ satisfies condition \eqref{KOKOU condition}. Note that if $\chi'$ is trivial on the index 2 subgroup $H$ but non-trivial on $2H$, then $\chi'$ takes values in $\{ +1,-1\}$. Upon closer inspection there is only one character that satisfies these properties; $\chi'$ has to be equal to 1 on $H$ and equal to $-1$ on $2H$. Thus the sum over primes $p \equiv 2 \bmod 3$ evaluates to
    \begin{align*}
    \allowdisplaybreaks
        \sums{
        p \leq x \\
        p \equiv 2 \bmod 3}
        \chi'(p)\log p
        &=
        \sums{\delta \bmod 27\\
        \delta \equiv 2 \bmod 3}
        \chi'(\delta)
        \sums{
        p \leq x \\
        p \equiv \delta \bmod 27}
        \log p
        \allowdisplaybreaks\\
        &=
        -9
        \left(\frac{1}{\phi(27)}x+\OO(x \exp( - C_W (\log x)^{\sfrac{1}{2}}))\right)
        \allowdisplaybreaks\\
        &=
        -\frac{1}{2}x+\OO(x \exp( - C_W (\log x)^{\sfrac{1}{2}})).
    \end{align*}
    Here we applied the Siegel--Walfisz theorem, and $C_W$ is some constant coming from this application. The sum over primes $p \equiv 1 \bmod 3$ and the sum over primes that divide $m$ can be treated the same as in i). Thus (for $Q$ as in i)) we obtain that for all $x \geq Q$ and $A'>0$,
    \begin{align*}
        \sum_{p \leq x}\chi'(p)g_m(p) \log p 
        &=
        \frac{1}{6}x - \frac{1}{2}x + \OO(x \exp( - C_{W} (\log x)^{\sfrac{1}{2}})) - \OO(\omega(m)\log m)
        \allowdisplaybreaks\\
        &= - \frac{1}{3}x + \OO\left(\frac{x}{(\log x)^{A'}}\right).
    \end{align*}
    In case that $\chi'$ is non-trivial on $G$ but trivial on $H$, applying Theorem \ref{Theorem 13.2 KOUKOU} (with $\alpha = -\sfrac{1}{3}$ and otherwise the same parameters as in i)) to the sum $\sum
    \chi'(n)g_m(n)$ yields
    \begin{align}
    \label{gmtrivialonH}
    \sums{
        n \leq B
    }
    \chi'(n)g_m(n)
    =
    \OO\bigg(\frac{B}{(\log B)^{\frac{4}{3}}}\bigg).
    \end{align}
    
    iii) If $\chi'$ is non-trivial on $H$, $\chi'(n)g_m(n)$ satisfies condition \eqref{KOKOU condition} with $\alpha = 0$, $k = 1$, $Q = \exp((\log B)^{\frac{1}{100}})$, $\epsilon = 1$ and $J=1$. We split the sum as we did in i),
    \begin{align*}
        \sum_{p \leq x}\chi'(p)g_m(p) \log p 
        &= 
        \sums{
        p \leq x \\
        p \equiv 1 \bmod 3} 
        \frac{1}{3} \chi'(p)\log p 
        + 
        \sums{
        p \leq x \\
        p \equiv 2 \bmod 3}
        \chi'(p)\log p 
        - 
        \sums{p \mid m} \frac{1}{3^{\omega_1(p)}} \chi'(p)\log p.
    \end{align*}
    However, $\chi'$ is now non-trivial and we get a cancellation of the main term due to orthogonality of characters. For the sum over $p \equiv 2 \bmod 3$ we get that for all $x \geq Q$ and $A' >0$,
    \begin{align*}
    \allowdisplaybreaks
        \sums{
        p \leq x \\
        p \equiv 2 \bmod 3}
        \chi'(p)\log p
        &=
        \sums{\delta \bmod 27\\
        \delta \equiv 2 \bmod 3}
        \chi'(\delta)
        \sums{
        p \leq x \\
        p \equiv \delta \bmod 27}
        \log p
        \allowdisplaybreaks\\
        &=
        \chi'(2)\sums{\xi \in H}\chi'(\xi)
        \left(\frac{1}{\phi(27)}x+\OO(x \exp( - C_W (\log x)^{\sfrac{1}{2}}))\right)
        \allowdisplaybreaks\\
        &=
        \OO(x \exp( - C_W (\log x)^{\sfrac{1}{2}}))
        \allowdisplaybreaks\\
        &=
        \OO\left(\frac{x}{(\log x)^{A'}}\right).
    \end{align*}
    The sum over $p \equiv 1 \bmod 3$ can be bounded analogously and the sum over $p \mid m$ fits into the same error term, as shown in i). Thus $\chi'(n)g_m(n)$ satisfies condition \eqref{KOKOU condition} with $\alpha = 0$. Applying Theorem \ref{Theorem 13.2 KOUKOU} (with parameters $\alpha = 0$, $k=1$, $Q = \exp((\log B)^{\frac{1}{100}})$, $\epsilon = 1$ and $J=1$) yields
    \begin{align}
    \label{gmnontrivial}
    \sums{
        n \leq B
    }
    \chi'(n)g_m(n)
    =
    \OO\left(\frac{B(\log Q)^{J+1}}{(\log B)^{J+1}}\right)
    =
    \OO\left(\frac{B}{(\log B)^{\frac{4}{3}}}\right).
    \end{align}
    Importantly we have followed the convention in Remark \ref{convention}.

    Combining \eqref{gmtrivial}, \eqref{gmtrivialonH} and \eqref{gmnontrivial}, we see that \eqref{Admissiblepulledout} becomes
    \begin{align}
    \allowdisplaybreaks
    \nonumber
    N(A,B)_{\substack{
    d_2=d_3= 1 \\
    e_2=e_3=1
    }}
    =
    \frac{1}{\phi(27)^2}&
    \sums{\delta, \epsilon \bmod 27\\
    (\delta, \epsilon) \text{ admissible}}
    \sum_{\chi, \chi' \in G}
    \chi(\delta^{-1})\chi'(\epsilon^{-1})
    \allowdisplaybreaks\\
    \label{gettingthere}
    &\sum_{
    m \leq A
    }
    \chi(m)g(m)
    \left(\mathbb{1}_{\{\chi' = \chi_0\}}\frac{B}{(\log B)^{\frac{1}{3}}\Gamma(\frac{2}{3})}c_0(m) + \OO\bigg(\frac{B}{(\log B)^{\frac{4}{3}}}\bigg)\right),
    \end{align}
    where $\mathbb{1}_{\{\chi' = \chi_0\}}$ is the indicator function of $\chi'$ being trivial and
    $$
    c_0(m) = \prod_{p}\left(1+\frac{g(p)}{p}\right)\left(1- \frac{1}{p}\right)^{\frac{2}{3}}\prod_{p\mid m}\left(1+\frac{g(p)}{p}\right)^{-1}.
    $$
    To evaluate the sum over $m$ in \eqref{gettingthere}, one can perform a calculation that is analogous to the above. Again one has to distinguish three cases, depending on the (non-)triviality of the character $\chi$, and use a combination of Siegel--Walfisz and Theorem \ref{Theorem 13.2 KOUKOU}. The main term comes from the summand where $\chi$ and $\chi'$ are both trivial. All other combinations of characters lead to smaller contributions that fit into the error term (as in cases ii) and iii) above), so the sum over $m$ becomes
    \begin{align}
    \allowdisplaybreaks
    \nonumber
        &\sum_{m \leq A}
        \chi(m)g(m)
        \left(\mathbb{1}_{\{\chi' = \chi_0\}}\frac{B}{(\log B)^{\frac{1}{3}}\Gamma(\frac{2}{3})}c_0(m) + \OO\bigg(\frac{B}{(\log B)^{\frac{4}{3}}}\bigg)\right)
        \allowdisplaybreaks\\
        &=
        \label{almostdone}
        \mathbb{1}_{\{\chi = \chi' = \chi_0\}}c_1\frac{AB}{(\log A)^{\frac{1}{3}}(\log B)^{\frac{1}{3}}(\Gamma(\frac{2}{3}))^2}
        +\OO\bigg(\frac{AB}{(\log A)^{\frac{1}{3}}(\log B)^{\frac{1}{3}}}\bigg(\frac{1}{\log A} + \frac{1}{\log B}\bigg)\bigg).        
    \end{align}
    Here $c_1$ is the constant
    $$
    c_1 = \prod_{p}\left(1+\frac{2}{3^{\omega_1(p)}p}\right)\left(1- \frac{1}{p}\right)^{\frac{4}{3}}.
    $$
    Substituting \eqref{almostdone} into \eqref{gettingthere} we finally obtain the contribution of case 4 to $N(A,B)$
    \begin{gather}
    \label{thisistheend}
    N(A,B)_{\substack{
    d_2=d_3= 1 \\
    e_2=e_3=1
    }}
    =
    c_2\frac{AB}{(\log A)^{\frac{1}{3}}(\log B)^{\frac{1}{3}}}
    \left(1+\OO\bigg(\frac{1}{\log A} + \frac{1}{\log B}\bigg)\right),
    \end{gather}
    where 
    $$
    c_2 = c_1\cdot \Big(\Gamma\Big(\frac{2}{3}\Big)\Big)^{-2} 
    = \Big(\Gamma\Big(\frac{2}{3}\Big)\Big)^{-2}\prod_{p}\left(1+\frac{2}{3^{\omega_1(p)}p}\right)\left(1- \frac{1}{p}\right)^{\frac{4}{3}}.
    $$
    
    We end the proof by choosing the parameters $V$, $C$ and $k$ in terms of $\kappa_1$, $\kappa_2$ and $\delta$, such that the contributions of cases one to three fit into the error term of the theorem, which is the error term in \eqref{thisistheend}. As remarked before, $k$ can be chosen as any integer larger than $2$, we choose $k = 4$. A calculation shows that 
    $$
    V = (\log AB)^{\frac{\kappa_2 + 4}{\kappa_1}} \quad \text{and} \quad C = 1 + \frac{1}{\delta}\Big(\frac{2\kappa_2 + 8}{\kappa_1}+\frac{5}{3}\Big)
    $$
    are sufficient to bound the contributions \eqref{Case1} and \eqref{Case3b} of cases 1 and 3b, and that the contributions \eqref{Case2} and \eqref{Case3a} of cases 2 and 3a  are compatible with these choices. Note that we have made extensive use of the assumption that $A,B \geq \exp ((\log AB)^{\delta})$ in the calculation of $V$ and $C$.

%% file: 5Appendix.tex
\section{List of admissible pairs and code}
\label{List of admissable pairs}
In the proof of Lemma \ref{p nmid ab} we use two small pieces of code to check for solutions of $\phi_{a,b}$ in $\mathbb{F}_p$ for $p = 2,5,7$ (where $p \nmid ab$) and in $\sfrac{\Z}{27\Z}$. The code was written in the IDE Juypter, based on Python.

The following code iterates over forms $\phi_{a,b} \bmod p$ and triples $(x,y,z) \bmod p$ until it finds a non-zero solution of $\phi_{a,b}$. We use this to confirm that all forms have a solution in $\mathbb{F}_p$ for $p = 2,5,7$.
\begin{lstlisting}[language=Python, basicstyle=\small, showstringspaces=false]
import itertools as it
Allsolvable = 0
p = 7   # can choose p=2,5,7 here

for a,b in it.product(range(1,p), repeat=2):
    for x,y,z in it.product(range(p), repeat=3):
        if  (x,y,z)!=(0,0,0) and (a*(x**3)+b*(y**3)-(z**3))%p == 0:
            Allsolvable += 1
            break

if Allsolvable == (p-1)**2:
    print('All forms have a non-zero solution mod',p,'.')
\end{lstlisting}
As shown in the proof of Lemma \ref{p nmid ab}, a form $\phi_{a,b}$ with coefficients $(a,b)$ not divisible by three has a solution in $\Q_3$ iff it has a solution $(x,y,z)\bmod 27$ that is not equivalent to $(0,0,0) \bmod 3$. The following code iterates over coefficient pairs $(a,b)$ and triples $(x,y,z)$ mod 27 to compile the list of pairs $(a,b)$ such that $\phi_{a,b}$ has a solution in $\Q_3$.
\begin{lstlisting}[language=Python, basicstyle=\small, showstringspaces=false]
import itertools as it
Admissible_pairs = []

for a,b in it.product(range(27), repeat=2):
    if a%3 == 0 or b%3 == 0:
        continue
    for x,y,z in it.product(range(27), repeat=3):
        if x%3 + y%3 + z%3 == 0:
            continue
        elif  (a*(x**3)+b*(y**3)-(z**3))%27 == 0:
            Admissible_pairs.append([a,b])
            break

print(Admissible_pairs)
\end{lstlisting}
The list produced by this code is given below. We call it the list of admissible pairs.
\begin{flushleft}
        \mylist
        (1,1) (1,2) (1,4) (1,5) (1,7) (1,8) (1,10) (1,11) (1,13) (1,14) (1,16) (1,17) (1,19) (1,20) (1,22) (1,23) (1,25) (1,26) (2,1) (2,2) (2,7) (2,8) (2,10) (2,11) (2,16) (2,17) (2,19) (2,20) (2,25) (2,26) (4,1) (4,4) (4,5) (4,8) (4,10) (4,13) (4,14) (4,17) (4,19) (4,22) (4,23) (4,26) (5,1) (5,4) (5,5) (5,8) (5,10) (5,13) (5,14) (5,17) (5,19) (5,22) (5,23) (5,26) (7,1) (7,2) (7,7) (7,8) (7,10) (7,11) (7,16) (7,17) (7,19) (7,20) (7,25) (7,26) (8,1) (8,2) (8,4) (8,5) (8,7) (8,8) (8,10) (8,11) (8,13) (8,14) (8,16) (8,17) (8,19) (8,20) (8,22) (8,23) (8,25) (8,26) (10,1) (10,2) (10,4) (10,5) (10,7) (10,8) (10,10) (10,11) (10,13) (10,14) (10,16) (10,17) (10,19) (10,20) (10,22) (10,23) (10,25) (10,26) (11,1) (11,2) (11,7) (11,8) (11,10) (11,11) (11,16) (11,17) (11,19) (11,20) (11,25) (11,26) (13,1) (13,4) (13,5) (13,8) (13,10) (13,13) (13,14) (13,17) (13,19) (13,22) (13,23) (13,26) (14,1) (14,4) (14,5) (14,8) (14,10) (14,13) (14,14) (14,17) (14,19) (14,22) (14,23) (14,26) (16,1) (16,2) (16,7) (16,8) (16,10) (16,11) (16,16) (16,17) (16,19) (16,20) (16,25) (16,26) (17,1) (17,2) (17,4) (17,5) (17,7) (17,8) (17,10) (17,11) (17,13) (17,14) (17,16) (17,17) (17,19) (17,20) (17,22) (17,23) (17,25) (17,26) (19,1) (19,2) (19,4) (19,5) (19,7) (19,8) (19,10) (19,11) (19,13) (19,14) (19,16) (19,17) (19,19) (19,20) (19,22) (19,23) (19,25) (19,26) (20,1) (20,2) (20,7) (20,8) (20,10) (20,11) (20,16) (20,17) (20,19) (20,20) (20,25) (20,26) (22,1) (22,4) (22,5) (22,8) (22,10) (22,13) (22,14) (22,17) (22,19) (22,22) (22,23) (22,26) (23,1) (23,4) (23,5) (23,8) (23,10) (23,13) (23,14) (23,17) (23,19) (23,22) (23,23) (23,26) (25,1) (25,2) (25,7) (25,8) (25,10) (25,11) (25,16) (25,17) (25,19) (25,20) (25,25) (25,26) (26,1) (26,2) (26,4) (26,5) (26,7) (26,8) (26,10) (26,11) (26,13) (26,14) (26,16) (26,17) (26,19) (26,20) (26,22) (26,23) (26,25) (26,26)  !
    \end{flushleft}
    There are 252 admissible pairs out of $18^2=324$ pairs $(a,b)$ mod 27 that satisfy $3\nmid ab$. For any $a \bmod 27$ there are either 12 or 18 admissible pairs of the form $(a, \cdot)$.